\documentclass[letterpaper, 10 pt, conference]{ieeeconf}  

\IEEEoverridecommandlockouts                              
\usepackage{color}
\usepackage[T1]{fontenc}
\usepackage{amsthm}
\usepackage{amsmath}
\usepackage{amssymb}
\usepackage{graphicx}
\usepackage[unicode=true,
 bookmarks=true,bookmarksnumbered=true,bookmarksopen=true,bookmarksopenlevel=1,
 breaklinks=false,pdfborder={0 0 0},backref=false,colorlinks=false]
 {hyperref}
\hypersetup{pdftitle={Your Title},
 pdfauthor={Your Name},
 pdfpagelayout=OneColumn, pdfnewwindow=true, pdfstartview=XYZ, plainpages=false}

\makeatletter
\theoremstyle{plain}
\newtheorem{thm}{\protect\theoremname}
\theoremstyle{plain}
\newtheorem{cor}[thm]{\protect\corollaryname}

\usepackage[caption=false,font=footnotesize]{subfig}

\makeatother

\providecommand{\corollaryname}{Corollary}
\providecommand{\theoremname}{Theorem}

\begin{document}

\title{Monotonicity of Dissipative Flow Networks Renders Robust Maximum
Profit Problem Tractable: General Analysis and Application to Natural
Gas Flows}


\author{Marc Vuffray, Sidhant Misra, Michael Chertkov
\thanks{S. Misra, M. Vuffray and M. Chertkov. are with the Center for Nonlinear Studies and Theoretical Division T-4 of Los Alamos National Laboratory,
        Los Alamos, NM 87544. M. Chertkov is also affiliated with the New Mexico Consortium, Los Alamos, NM 87544.
        {\tt\small \{vuffray|sidhant|chertkov\}@lanl.gov}}
}

\maketitle
\begin{abstract}
We consider general, steady, balanced flows of a commodity over a network where an instance of the network flow is characterized by edge flows and nodal potentials. Edge flows in and out of a node are assumed to be conserved,  thus representing standard network flow relations. The remaining freedom in the flow distribution over the network is constrained by potentials so that the difference of potentials at the head and the tail of an edge is expressed as a nonlinear function of the edge flow. We consider networks with nodes divided into three categories: sources that inject flows into the network for a certain cost, terminals which buy the flow at a fixed price
and ``internal'' customers each withdrawing an uncertain amount of flow, which has a priority and thus it is not priced. Our aim is to operate the network such that the profit, i.e. amount of flow sold to terminals minus cost of injection, is maximized, while maintaining the potentials within prescribed bounds. We also require that the operating point is robust with respect to the uncertainty of customers' withdrawals. In this setting we prove that potentials are monotonic functions of the withdrawals. This observation enables us to replace in the maximum profit optimization infinitely many nodal constraints, each representing a particular value of withdrawal uncertainty, by only two constraints representing the cases where all nodes with uncertainty consume their minimum and maximum amounts respectively. We illustrate this general result on example of the natural gas transmission network. In this enabling example gas withdrawals by consumers are assumed uncertain, the potentials are gas pressures squared, the potential drop functions are bilinear in the flow and its intensity with an added tunable factor representing compression.
\end{abstract}

\IEEEpeerreviewmaketitle{}

\section{Introduction\label{sec:Introduction}}

The maximum profit, or alternatively minimum loss, network flow problem aims to maximize monetary benefit by delivering maximum amount of flow (of a commodity) from sources to terminals. The setting is general, and as such it applies to natural gas networks \cite{Osiadacz_1987,12BNV,Misra2015},  our enabling example,  but also to electric circuits \cite{Kirchhoff1847,98Bol} and traffic flows \cite{10CSADF,13Var_a,13Var_b}.

Formally, these problems are constructed by extending the standard network flow setting, see e.g. \cite{89AMO,95AMOR} and references there in, with additional physical constraints, introducing nodal potentials and relating the potential drop along an edge of the network to a function of the flow.  Thus,  in the case of the gas flows,  the potentials are pressures squared and the pressure square drop is a bilinear function of the flow  and the flow amplitude with added term related to compression \cite{Osiadacz_1987,12BNV,Misra2015}.


In this manuscript we focus on discussing a robust version of the maximum profit network flow problem, more accurately ``adjustable robust optimization'' problem following the terminology commonly accepted in the literature on robust optimization \cite{03BS,ben2009robust,13BNS,13BG}. This means that in the robust optimization network flow model considered there are three different types of variables: the uncertain variables, the non-adjustable variables and the adjustable variables. The uncertain variables express
information that is not certain, i.e. available for the optimization decision only in the form of allowed range. The non-adjustable variables represent the ``here and now'' decision in the system. Their values should be feasible for any realization of the uncertain variables from the allowed range. Finally the adjustable variables represents the ``wait and see'' decisions. Their values
are adaptable to a particular values of the uncertain parameters.

The adjustable robust optimization is composed of a test of robust feasibility and an optimization procedure. A value of the non-adjustable variables is said to be robust feasible if for any acceptable configuration of the uncertain variables there exists feasible values of the adjustable variables. Then the optimization procedure consists in finding a robust feasible configuration of the non-adjustable variables such that the objective function, maximum profit or minimum loss, is minimized.

The difficulty in solving robust optimization problems arises primarily due to the robust feasibility constraint, which is in essence an intersection of  infinitely (and possibly uncountably) many constraints, one corresponding to each
allowed value of the uncertain parameters. This results in the so-called semi-infinite program \cite{Hettich93}. In the robust optimization literature, the ways to handle these constraints can be classified into three different categories. First, when the constraints and the uncertainty set have special structure, e.g., linear constraints and ellipsoidal uncertainty set, it is possible to  use duality theory to represent the infinitely many constraints with one single  dual feasibility constraint \cite{BentalNemirovski98},
\cite{Nemirovski99}, \cite{BenTal00}. Also included in this category is approximations and relaxations of more complicated uncertainly sets and/or constraints with simpler ones that are amenable to duality theory.
The second category/approach is similar to the so-called ``scenario based'' approach, where  a (possibly random) sampling of the uncertain parameters is performed, and the feasibility constraint corresponding to each sampled parameter is included in the optimization formulation \cite{Bertsimas07}, \cite{Boyd09}. The quality of the solution thus obtained depends on the number of samples used and also on how the samples were chosen. The third case, which is the approach taken in this manuscript, is when one can analytically or numerically identify the ``extreme-cases'', i.e, find the subset of values of the uncertain parameters that can violate the feasibility constraints. When this subset is finite, or has a finite representation, the robust feasibility constraint again reduces to a finite number of standard constraints. Examples where this strategy is used are scarce. (See \cite[pp. 388]{Boyd09} for a brief discussion on the topic.)

Our main result, stated in Theorem \ref{thm:ARC_max_throughput}, is that the adjustable robust maximum profit problem is tractable,  in the sense that instead of keeping infinitely many conditions, associated with all possible realizations of the uncertain variables, it is sufficient to only account for two extremal conditions correspondent to every uncertain variable (customer withdrawal) greedily maximize/minimize their values. We implement a multi-stage strategy to prove the results. First, we prove in Theorem \ref{thm:uniqueness_existence_solutions} existence and uniqueness of the optimal solutions for the adjustable variables given that the other parameters in the problem are fixed. Then, we prove in Theorem \ref{thm:aquarius} that the adjustable variables are monotonic functions of the uncertain variables. Finally, we combine all these result to prove the main tractability statement. Schematic outline of our proof strategy is shown in Fig.~\ref{fig:proof_strategy}.
\begin{figure}[tbh]
\centering{}\includegraphics[scale=1.0]{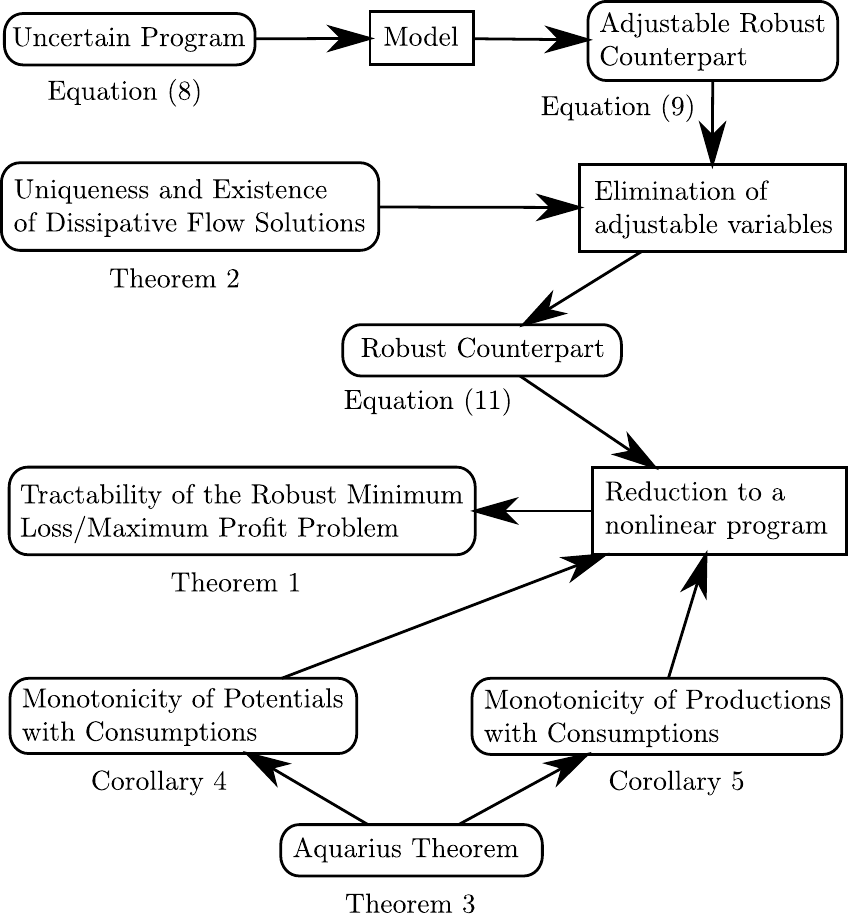}\protect\caption{\label{fig:proof_strategy}
Schematic diagram relating different statements of the manuscript.}
\end{figure}

The manuscript is organized as follows. In Section \ref{sec:setting-and-result},
we briefly introduce the concept of dissipative network flows and
define the problem of the adjustable robust maximum profit/minimum loss problem.
In this Section we also state our main theorem/statement. In Section \ref{sec:energy-function-method}
we prove existence and uniqueness of a configuration of the adjustable
variables given that the other parameters are fixed. The proof is based on the energy function method which consist in rewriting a set
of equations as the result of a convex minimization problem. In Section \ref{sec:monotonicity-property} we prove that the adjustable variables are monotonic with the respect to uncertain variables. This result
is general -- it applies to the case of unconstrained network flows. Finally, in Section \ref{sec:Illustration-with-Natural-gas-network},
we discuss application of our results to robust optimization in natural gas transportation networks.

\section{Formulation and the Main Result \label{sec:setting-and-result}}

\subsection{Dissipative Flow Networks}

Let $G=\left(V,E\right)$ be a connected directed graph, where $V$ and $E\subset V\times V$ denotes the set of nodes and directed edges respectively.   Our notational convention is that for any pair of , $i$ and $j$, connected by an edge,  both $(i,j)$ and $(j,i)$ belong to $E$.
Function $q:V\rightarrow\mathbb{R}$ associates production/consumption $q_i$ to the node $i\in V$.
We assume that the total production/consumption is balanced
\begin{equation}
\sum_{i\in V}q_{i}=0.\label{eq:production_balance}
\end{equation}

Flow $\phi:E\rightarrow\mathbb{R}$ is a function mapping an edge
$\left(i,j\right)\in E$ to flow intensity $\phi_{ij}$ (which we also call flow, or edge flow, when not confusing). The  flow is skew-symmetric with respect to the edge inversion,
i.e. for all $\left(i,j\right)\in E$
\begin{equation}
\phi_{ij}=-\phi_{ji}.\label{eq:flow_anti_symmetry}
\end{equation}
The flow is conserved locally,  i.e. at any  $i\in V$
equation
\begin{equation}
\sum_{j\in\partial i}\phi_{ji}+q_{i}=0,\label{eq:flow_conservation}
\end{equation}
where we denote the set of neighbors of $i\in V$ by $\partial i:=\left\{ j\in V\mid\left(j,i\right)\in E\right\} $.
Note that in the context of electric circuits, the flow conservation
equation \eqref{eq:flow_conservation} is also called the first Kirchhoff's
Law.

Flow over an edge $(i,j)\in{\cal E}$ is related to the potentials at end-nodes of the edges through the following potential drop equation
\begin{equation}
\pi_{j}-\pi_{i}=-f_{ij}\left(\phi_{ij}\right),\label{eq:potential_drop}
\end{equation}
where the functions, $f_{ij}(\cdot)$, coined dissipation functions, are assumed to be continuous and monotonically decreasing, thus leading to decrease/dissipation of potential along the direction of the flow.
The dissipation functions satisfy the following symmetry relation with respect to an edge inversion
\begin{equation}
f_{ij}\left(x\right)=-f_{ji}\left(-x\right),\label{eq:drop_symmetry}
\end{equation}
and they also possess an operational freedom, i.e. the dissipation functions can take any form from their admissible set, $f\in \mathcal{F}$ depending on the problem considered.


Note that existence of a potential function that satisfies Eq.~\eqref{eq:potential_drop}
allows a graph cycle interpretation.
Let $\gamma=\left\{ i_{1},\ldots,i_{n}\right\} \in V^{n}$
be a directed path in $G$ such that $i_{k}\neq i_{l}$ for $k,l\in\left\{ 1,...n-1\right\} $.
If $i_{1}=i_{n}$ we say that $\gamma$ is a directed cycle, otherwise
we called $\gamma$ a non-intersecting directed path. For any directed
cycle $\gamma$ in $G$ we sum the potential drop Eq.~\eqref{eq:potential_drop}
along the directed edges of $\gamma$ to get the following expression
\begin{equation}
\sum_{k=1}^{n-1}f_{i_{k}i_{k+1}}\left(\phi_{i_{k}i_{k+1}}\right)=0.\label{eq:second_Kirchhoff_law}
\end{equation}
Eqs. \eqref{eq:second_Kirchhoff_law} can be interpreted as a nonlinear generalization of what is called second Kirchoff's law in the circuit theory. In fact, conditions expressing the fact that the flows satisfy
Eqs.~\eqref{eq:second_Kirchhoff_law} for any cycle basis of the graph is fully equivalent to the set of conditions \eqref{eq:potential_drop} over all edges of the graph.

We call the network $G=\left(V,E\right)$,  with flows, potentials and the dissipative functions, $f$, defined on it, the dissipative flow network.

\subsection{Robust Maximum Profit/Minimum Loss Optimization}

We consider a dissipative flow network where the set of nodes
is partitioned into three non-empty subsets $V=S\cup T\cup R$. The
set $S$ contains nodes $i\in S$ that are sources injecting flow into the network which costs $g_{i}\left(q_{i}\right)$,
where $g_{i}$ is a pre-set function. Nodes $i\in T$ are terminals where the flows are withdrawn leading to the following payment per node, $h_{i}\left(q_{i}\right)$,
where $h_{i}$ is non-decreasing function. Nodes $i\in R$ are internal
customers each consuming a fixed amount of flow that is uncertain, i.e. it lies
in the set $\left[\underline{q}_{i},\overline{q}_{i}\right]$. We assume that the internal consumption is mandatory, and as such it is not priced. We aim to maximize the revenue, defined as a difference between the profit brought by selling the flow at the terminals minus the cost of injection. Alternatively we can state the problems as minimization of expenses, defined as the cost of injection minus the profit at the terminal:
\begin{equation}
c\left(q_{S},q_{T}\right):=\sum_{i\in S}g_{i}\left(q_{i}\right)-\sum_{i\in T}h_{i}\left(q_{i}\right),\label{eq:max_throughput_objective}
\end{equation}
where $q_{S}$ and $q_{T}$ are shorthand notations for  production
at the sources and the terminals respectively.

The maximum profit (minimum loss) problem is an uncertain program that minimize
the cost (\ref{eq:max_throughput_objective}) with respect to injections, potentials, flows and dissipation functions from their admissible set $\mathcal{F}$
\begin{equation}
\begin{aligned} & \underset{q,\pi, f,\phi}{\text{minimize}} &  & c\left(q_{S},q_{T}\right)\\
 & \text{subject to} &  & q_{i}=q_{i}^{\text{realization}} & \forall i\in R,\\
 &  &  & \sum_{j\in\partial i}\phi_{ji}+q_{i}=0 & \forall i\in V\\
 &  &  & \pi_{j}-\pi_{i}=-f_{ij}\left(\phi_{ij}\right) & \forall\left(i,j\right)\in E\\
 &  &  & \sum_{i\in V}q_{i}=0\\
 &  &  & \underline{\pi}_{i}\leq\pi_{i}\leq\overline{\pi}_{i} & \forall i\in V,
\end{aligned}
\label{eq:uncertain_max_throughput}
\end{equation}
where the regular customer withdrawals could be any, but fixed, value $q_{i}^{\text{realization}}\in\left[\underline{q}_{i},\overline{q}_{i}\right]$.

The optimization (\ref{eq:uncertain_max_throughput}) is deterministic as the uncertain parameters are fixed.  Below we will turn to its robust equivalent, thus aiming at taking optimization decision leading to a solution which is feasible for any realization of uncertainty while also achieving losses which are first maximized (adversarially) over the internal (uncertain) degrees of freedom and then minimized over the external (operational) degrees of freedom.
Therefore, to arrive at a plausible robust optimization formulation one, first, need to split all the variables into the sets of the so-called operational (non-adjustable) and adjustable variables \cite{ben2009robust}.
Operational variables represent decisions taken regardless of a particular configuration of the uncertainty.
Adjustable variables, according to the name, are flexible, i.e. they adjust themselves to the uncertainty.


We consider the case where the operational variables are the production
values at the source $q_{S}:=\left\{ q_{i}\right\} _{i\in S}$, the
potential values at the terminals $\pi_{T}:=\left\{ \pi_{i}\right\} _{i\in T}$
and the dissipation functions $f\in\mathcal{F}$. The adjustable variables
are the production values at the terminals $q_{T}:=\left\{ q_{i}\right\} _{i\in T}$,
the potential values at the source $\pi_{S}:=\left\{ \pi_{i}\right\} _{i\in S}$,
the potential values at the consumers $\pi_{R}:=\left\{ \pi_{i}\right\} _{i\in R}$
and the flows, $\phi$. The production values at the consumers are denoted by $q_{R}:=\left\{ q_{i}\right\} _{i\in R}$ and the cartesian product of their uncertainty set is denoted as $Q=\left[\underline{q}_{i_{1}},\overline{q}_{i_{1}}\right]\times\cdots\times\left[\underline{q}_{i_{\left|R\right|}},\overline{q}_{i_{\left|R\right|}}\right]$.
We refer the reader to the illustration of the natural gas network in Section
\ref{sec:Illustration-with-Natural-gas-network} to justify this particular choice of adjustable and non-adjustable variables.

With this choice of the operational variables, an adjustable robust counterpart
of the deterministic minimum cost problem \eqref{eq:uncertain_max_throughput} becomes
\begin{equation}
\begin{aligned} & \underset{q_{S},\pi_{T},f,x}{\text{minimize}} &  & x\\
 & \text{such that } &  & \forall q_{R}\in Q,\,\exists q_{T},\pi_{S},\pi_{R},\phi\\
 & \text{subject to} &  & c\left(q_{S},q_{T}\right)\leq x\\
 &  &  & \sum_{j\in\partial i}\phi_{ji}+q_{i}=0 & \forall i\in V\\
 &  &  & \pi_{j}-\pi_{i}=-f_{ij}\left(\phi_{ij}\right) & \forall\left(i,j\right)\in E\\
 &  &  & \sum_{i\in V}q_{i}=0\\
 &  &  & \underline{\pi}_{i}\leq\pi_{i}\leq\overline{\pi}_{i} & \forall i\in V.
\end{aligned}
\label{eq:ARC_max_throughput}
\end{equation}
Considered directly,  Eq.~(\ref{eq:ARC_max_throughput}) is intractable as containing an infinite number of constraints,
each associated with a particular $q_{R}\in Q$.
However,  and in spite of the grim naive assessment, Eq.~(\ref{eq:ARC_max_throughput}) allows tractable re-formulation stated in the following main theorem/result of the manuscript:
\begin{thm}[Tractability of the Robust Minimum Loss/Maximum Profit Problem]
\label{thm:ARC_max_throughput}There exists a function $\widetilde{q}^{T}:\mathbb{R}^{\left|R\right|+\left|S\right|+\left|T\right|}\times\mathcal{F}\rightarrow\mathbb{R}^{\left|T\right|}$
and a function $\widetilde{\pi}^{R\cup S}:\mathbb{R}^{\left|R\right|+\left|S\right|+\left|T\right|}\times\mathcal{F}\rightarrow\mathbb{R}^{\left|R\cup S\right|}$
such that the adjustable robust counterpart \eqref{eq:ARC_max_throughput}
reads
\[
\begin{aligned} & \underset{q_{S},\pi_{T},f}{\text{minimize}} &  & c\left(q_{S},\tilde{q}_{T}^{T}\left(\overline{q}_{R},q_{S},\pi_{T},f\right)\right)\\
 & \text{subject to} &  & \underline{\pi}_{i}\leq\widetilde{\pi}_{i}^{R\cup S}\left(\underline{q}_{R},q_{S},\pi_{T},f\right) & \forall i\in R\cup S\\
 &  &  & \widetilde{\pi}_{i}^{R\cup S}\left(\overline{q}_{R},q_{S},\pi_{T},f\right)\leq\overline{\pi}_{i} & \forall i\in R\cup S\\
 &  &  & \underline{\pi}_{i}\leq\pi_{i}\leq\overline{\pi}_{i} & \forall i\in T,
\end{aligned}
\]
where $\underline{q}_{R}:=\left\{ \underline{q}_{i}\right\} _{i\in R}$
and $\overline{q}_{R}:=\left\{ \overline{q}_{i}\right\} _{i\in R}$
are respectively the minimum and maximum withdrawal values at the
consumer nodes.
\end{thm}
The functions $\widetilde{q}^{T}$ and $\widetilde{\pi}^{R\cup S}$
mentioned in Theorem \ref{thm:ARC_max_throughput} are derived explicitly, moreover stated in terms of the strictly convex optimizations, in Section \ref{sec:energy-function-method}. In words Theorem \ref{thm:ARC_max_throughput}
tells us that only the two extreme realizations of $q_{R}$ have to
be considered in order to solve the semi-infinite optimization \eqref{eq:ARC_max_throughput}.

Our strategy to prove the Theorem \ref{thm:ARC_max_throughput} is the
following: We first prove that there exists a unique solution of the
adjustable variables $q_{T}$, $\pi_{S}$, $\pi_{R}$ and $\phi$
given a configuration $q_{R}$, $q_{S}$, $\pi_{T}$ and $f$. This
states the existence of the functions $\widetilde{q}^{T}$ and $\widetilde{\pi}^{R\cup S}$.
Then, we show that the solution $q_{T}$, $\pi_{S}$, $\pi_{R}$
are monotonic functions of $q_{R}$, that enables us to reduce number of critical constraints to only two corresponding to the extreme configurations of the uncertainty. This establishes that $\widetilde{q}^{T}$
and $\widetilde{\pi}^{R\cup S}$ are monotonic functions of $q_{R}$.

The proof of existence and uniqueness of the solution is done in Section
\ref{sec:energy-function-method} using the energy function method.
The monotonicity properties are proved in Section \ref{sec:monotonicity-property}
and only require some properties of the flow network.

\section{Uniqueness of Dissipative Flow Solutions: the Energy Function Method\label{sec:energy-function-method}}

In this Section we show the existence and uniqueness of a configuration
of adjustable variables, given a configuration of the non-adjustable
variables. 
The key idea here is to relate solutions of the dissipative flow network equations \eqref{eq:production_balance},
\eqref{eq:flow_conservation} and \eqref{eq:potential_drop} to the extremum
of some convex ``energy'' function.
\begin{thm}[Uniqueness and Existence of the Dissipative Flow Solutions]
\label{thm:uniqueness_existence_solutions}Let $G=\left(V,E\right)$
be a dissipative flow network. Given $q_{R}$, $q_{S}$, $\pi_{T}$
and $f$, there exists a unique solution $q_{T}$, $\pi_{S}$, $\pi_{R}$
and $\phi$ which satisfies the balanced production equation \eqref{eq:production_balance},
the flow conservation equation \eqref{eq:flow_conservation} and the
potential drop equation \eqref{eq:potential_drop}.\end{thm}
\begin{proof}
Let us first invert the potential drop Eq.~\eqref{eq:potential_drop}
to express the flow with respect to the potential
\[
f_{ij}^{-1}\left(\pi_{i}-\pi_{j}\right)=\phi_{ij}.
\]
This operation is possible because $f_{ij}$ is an increasing function. Then, we
rewrite the flow conservation equations only in terms of the potential
\[
\sum_{j\in\partial i}f_{ji}^{-1}\left(\pi_{j}-\pi_{i}\right)+q_{i}=0.
\]
Therefore if one finds $\pi_{S}$ and $\pi_{R}$ that satisfy the flow conservation
equation for $i\in S\cup R$, one reconstructs the production $q_{T}$
using flow conservation for $i\in T$. It is easy to see that the
balanced production Eq.~\eqref{eq:production_balance} is automatically
satisfied
\begin{eqnarray*}
\sum_{i\in V}q_{i} & = & -\sum_{i\in V}\sum_{j\in\partial i}f_{ji}^{-1}\left(\pi_{j}-\pi_{i}\right)\\
 & = & -\sum_{\left(i,j\right)\in E}f_{ij}^{-1}\left(\pi_{i}-\pi_{j}\right)+f_{ji}^{-1}\left(\pi_{i}-\pi_{j}\right)\\
 & = & 0,
\end{eqnarray*}
where in the last line we use the symmetry relation \eqref{eq:drop_symmetry}
of the drop function.

We use the energy function method to prove that the potentials $\pi_{S}$
and $\pi_{R}$ are uniquely determined. Let us introduce
the set of oriented edges
\[
O:=\left\{ \left(i,j\right)\in E\mid i<j\right\} .
\]
The graph $\Gamma=\left(V,O\right)$ is an orientation of $G$ i.e. $\left(i,j\right) \in \Gamma$ if and only if $\left(i,j\right) \in G$ and $\left(j,i\right) \notin \Gamma$.  

Consider the following energy function
\begin{equation}
\mathcal{E}\left(\pi_{S},\pi_{R}\mid q_{R},q_{S},\pi_{T},f\right)=\sum_{\left(i,j\right)\in O}\varPsi_{ij}\left(\pi_{i}-\pi_{j}\right)-\sum_{i\in V}\pi_{i}q_{i},\label{eq:energy_function}
\end{equation}
where $\varPsi_{ij}$ is a primitive of $f_{ij}^{-1}$. Note that
the primitive exists because $f_{ij}^{-1}$ is continuous as $f_{ij}$
is continuous. It is easy to see that a minimum of $\mathcal{E}$
with respect to $\pi_{S}$ and $\pi_{R}$ satisfies the flow conservation
equation at $i\in S\cup R$. Indeed by taking the derivative, one arrives at
\[
\frac{\partial\mathcal{E}}{\partial\pi_{i}}=\sum_{j\in\partial i}f_{ji}^{-1}\left(\pi_{j}-\pi_{i}\right)+q_{i}.
\]
We now have to prove that the extremum is unique and that the extremum  is actually the minimum.

Let us define for every edge $\left(i,j\right)\in O$, variables
\[
\Delta_{ij}:=\pi_{i}-\pi_{j}.
\]
It is straightforward to see that the function
\[
\sum_{\left(i,j\right)\in O}\varPsi_{ij}\left(\Delta_{ij}\right),
\]
is a strictly convex function of $\Delta_{ij}$. Note that $f_{ij}^{-1}$
is increasing because $f_{ij}$ is increasing. This implies that $\Psi_{ij}$
is also a strictly convex function of $\Delta_{ij}$. Now we have to prove
that it the function is also strictly convex in $\pi_{S}$ and $\pi_{R}$.

First we express the relation between $\Delta_{ij}$ and $\pi_{i}$
in a matrix form using the incidence matrix $M$ of $\Gamma$. The
incidence matrix $M$ is a $\left|V\right|\times\left|O\right|$ matrix
with entries $M_{k,\left(i.j\right)}=\delta_{kj}-\delta_{ki}$ where
$\delta$ is the Kronecker's symbol. The relation between $\Delta=\left\{ \Delta_{ij}\right\} _{ij\in O}$
and $\pi$ reads
\[
\Delta=-M^{\intercal}\pi.
\]
We can also explicitly separate contributions of the potential
$\pi_{S}$ and $\pi_{R}$ from $\pi_{T}$ in this equation by
introducing the reduced incidence matrices $M_{V\setminus T}$ and
$M_{T}$. The matrices $M_{V\setminus T}$ and $M_{T}$ are equal
to $M$ with the lines, corresponding respectively to to $i\in T$
and $i\in V\setminus T$ respectively, removed. One finally obtains the following
equation
\[
\Delta=-M_{V\setminus T}^{\intercal}\left(\begin{array}{c}
\pi_{S}\\
\pi_{R}
\end{array}\right)-M_{T}^{\intercal}\pi_{T}.
\]
Observe that the matrix $M_{V\setminus T}^{\intercal}$ is full rank.
To see this, note that the product $M_{V\setminus T}M_{V\setminus T}^{\intercal}$
corresponds to the Laplacian of the graph $L=MM^{\intercal}$ with the lines $i$ and $j$, corresponding to $i,j\in T$, removed. Due to the Kirchoff's matrix-tree theorem \cite{Kirchhoff1847} we know that
the determinant of any co-factor of $L$ is non-zero on a connected
graph, which implies that
\[
\text{rank}\left(M_{V\setminus T}^{\intercal}\right)=\left|V\setminus T\right|.
\]
This proves that $\mathcal{E}$ is strictly convex in $\pi_{S}$ and
$\pi_{R}$.
\end{proof}
Theorem \ref{thm:uniqueness_existence_solutions} enables us to eliminate
the adjustable variables from the adjustable robust counterpart \eqref{eq:ARC_max_throughput}.
We can now rewrite our optimization problem as follows
\begin{equation}
\begin{aligned} & \underset{q_{S},\pi_{T},f,x}{\text{minimize}} &  & x\\
 & \text{such that } &  & \forall q_{R}\in Q,\\
 & \text{subject to} &  & c\left(q_{T},\widetilde{q}_{T}^{T}\left(q_{R},q_{S},\pi_{T},f\right)\right)\leq x\\
 &  &  & \underline{\pi}_{i}\leq\widetilde{\pi}_{i}^{R\cup S}\left(q_{R},q_{S},\pi_{T},f\right)\leq\overline{\pi}_{i} & \forall i\in R,S\\
 &  &  & \underline{\pi}_{i}\leq\pi_{i}\leq\overline{\pi}_{i} & \forall i\in T.
\end{aligned}
\label{eq:ARC_max_throughput_unique}
\end{equation}
Function $\widetilde{\pi}^{R\cup S}\left(q_{R},q_{S},\pi_{T},f\right)$
which outputs the potential $\pi_{i}$ for $i\in R\cup S$ was explicitly
constructed in the result of the strictly convex optimization
\begin{equation}
\widetilde{\pi}^{R\cup S}\left(q_{R},q_{S},\pi_{T},f\right):=\arg\min_{\pi_{S},\pi_{R}}\mathcal{E}\left(\pi_{S},\pi_{R}\mid q_{R},q_{S},\pi_{T},f\right),\label{eq:pressure_from_energy}
\end{equation}
where $\mathcal{E}$ is defined by \eqref{eq:energy_function}. Once
the potentials are found using Eq.~\eqref{eq:pressure_from_energy},
one can easily reconstruct the productions
\begin{equation}
\widetilde{q}_{i}^{T}\left(q_{R},q_{S},\pi_{T},f\right)=\sum_{j\in\partial i}f_{ji}^{-1}\left(\pi_{j}-\pi_{i}\right).\label{eq:production_from_energy}
\end{equation}

\section{Monotonicity Properties of the Dissipative Flow Networks\label{sec:monotonicity-property}}

The existence and uniqueness properties from Section \ref{sec:energy-function-method}
allowed us to simplify the robust minimum loss optimization problem. However, it still remains
in the form of an intractable semi-infinite program given by Eq.~\eqref{eq:ARC_max_throughput_unique}. In this Section
we show that only two scenarios for $q_{R}$ has to be considered
in order to solve Eq.~\eqref{eq:ARC_max_throughput_unique}. To achieve this goal we show that the potentials $\widetilde{\pi}^{R\cup S}$, given by Eq.~\eqref{eq:pressure_from_energy}, and the productions $\widetilde{q}_{i}^{T}$,
given by Eq.~\eqref{eq:production_from_energy} are monotonic functions of $q_{R}$.

\subsection{Flow Networks}

The following theorem applies to general network flows, i.e. flow networks which are not necessary
dissipative. We only require here that the flows satisfy the flow conservation
Eq.~\eqref{eq:flow_conservation}. This very general result will then become a starting point to prove monotonicity properties of the dissipative flow networks.
\begin{thm}[Aquarius Theorem]
\label{thm:aquarius}Let $G=\left(V,E\right)$ be a flow network
and let $\phi$ and $\phi^{*}$ be flows that satisfy the flow conservation
Eq.~\eqref{eq:flow_conservation} for the productions $q$ and
$q^{*}$ respectively. Let $T\subset V$ be a subset of $V$. If $q_{i}\geq q_{i}^{*}$
for all $i\in V\setminus T$, then for every node $u\in V\setminus T$
there exists a non-intersecting path $\left\{ i_{1},\ldots,i_{n}\right\} $
such that $i_{1}\in T$, $i_{n}=u$ and $\phi_{i_{l}i_{l+1}}^{*}\geq\phi_{i_{l}i_{l+1}}$.
Moreover if $q_{u}>q_{u}^{*}$ the inequality is strict i.e. $\phi_{i_{l}i_{l+1}}^{*}>\phi_{i_{l}i_{l+1}}$.\end{thm}
\begin{proof}
We construct the path by induction. Choose $u\in V\setminus T$ and
assume that $q_{u}>q_{u}^{*}$. The proof in the case where $q_{u}=q_{u}^{*}$
is identical. We define a sequence of subsets of nodes $B_{k}\subset V$
and $A_{k}=\bigcup_{l=1}^{k}B_{l}$ in the following way
\begin{eqnarray*}
B_{1} & := & \left\{ u\right\} \\
B_{k+1} & := & \left\{ i\in V\setminus A_{k}\mid\exists j\in\partial i\cap B_{k}\,\text{s.t.}\,\phi_{ij}^{*}>\phi_{ij}\right\} .
\end{eqnarray*}
An example of the sets $A_{k}$ and $B_{k}$ is shown in Fig.~\ref{fig:Construction-of-Ak-ensembles}.
\begin{figure}[tbh]
\centering{}\includegraphics[scale=0.5]{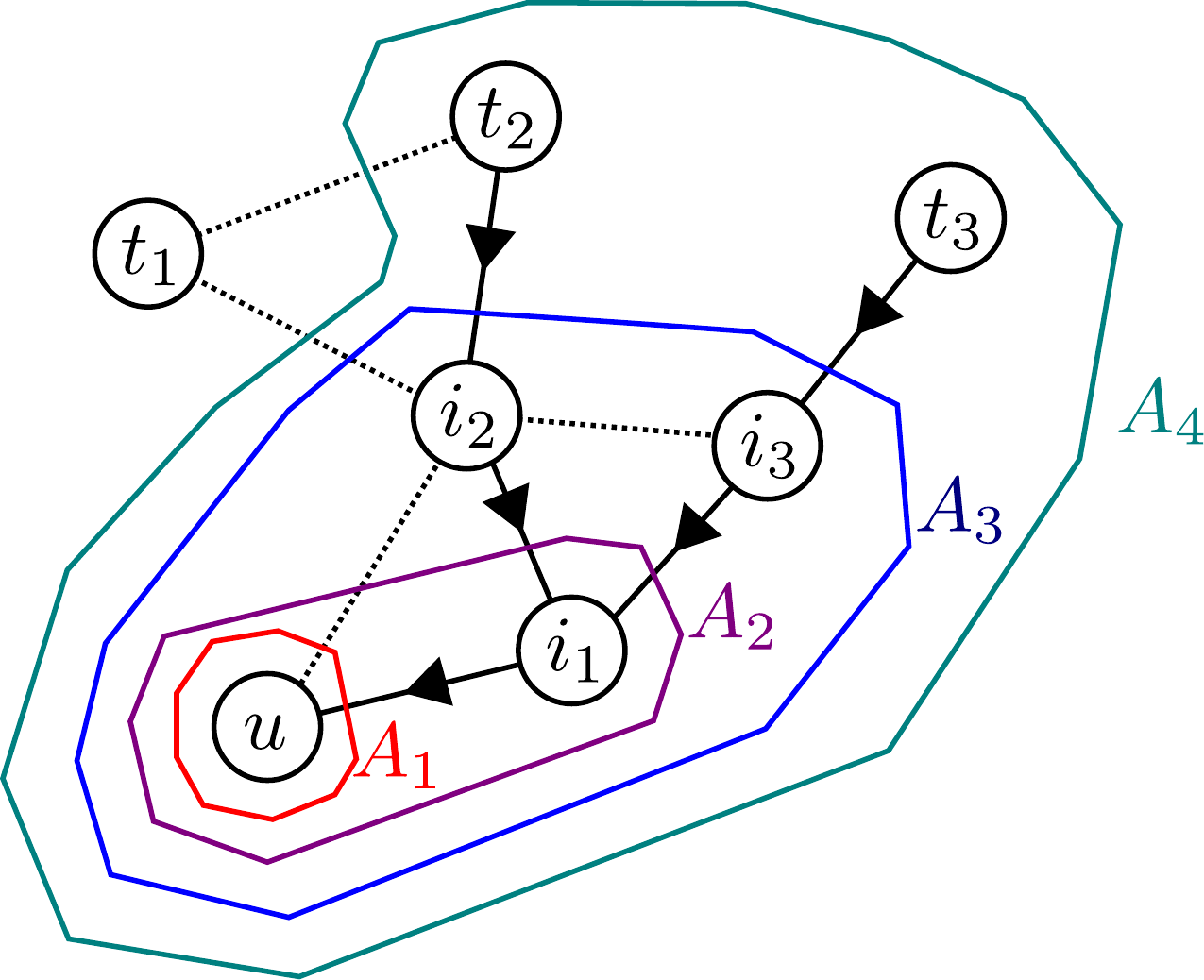}\protect\caption{\label{fig:Construction-of-Ak-ensembles}
Construction of the ensembles
$A_{k}$. Note that $B_{k}=A_{k}\setminus A_{k-1}$. The edges for
which $\phi_{ij}^{*}>\phi_{ij}$ with $i\in B_{k+1}$ and $j\in B_{k}$
are shown with a solid line and a dark arrow. }
\end{figure}

Let $n$ be the first $k$ for which $B_{k}\cap T\neq\emptyset$. First, we prove by induction that the sets $B_{k}$ for $k\leq n$ are non-empty. The set $B_{1}$ is non-empty by construction. Suppose
that the sets $B_{k}$ are non-empty. Define the set of edges connecting
two vertices in $A_{k}$ by
\[
E_{k}:=\left\{ \left(i,j\right)\in E\mid i,j\in A_{k}\right\} ,
\]
 and the set of edges with a starting point in $V\setminus A_{k}$
and an endpoint in $A_{k}$
\[
\overrightarrow{\partial}A_{k}:=\left\{ \left(j,i\right)\in E\mid j\notin A_{k}\,\text{and}\,i\in A_{k}\right\} .
\]
By summing the continuity Eq.~\eqref{eq:flow_conservation} for
the flow $\phi$ over the vertices in $A_{k}$ one obtains
\begin{eqnarray*}
0 & = & \sum_{i\in A_{k}}\left(q_{i}+\sum_{j\in\partial i}\phi_{ji}\right)\\
 & = & \sum_{i\in A_{k}}q_{i}+\sum_{i\in A_{k}}\sum_{j\in\partial i}\phi_{ji}\\
 & = & \sum_{i\in A_{k}}q_{i}+\sum_{\left(i,j\right)\in E_{k}}\left(\phi_{ij}+\phi_{ji}\right)+\sum_{\left(j,i\right)\in\overrightarrow{\partial}A_{k}}\phi_{ji}\\
 & = & \sum_{i\in A_{k}}q_{i}+\sum_{\left(j,i\right)\in\overrightarrow{\partial}A_{k}}\phi_{ji},
\end{eqnarray*}
where the skew-symmetry of the flow \eqref{eq:flow_anti_symmetry}
between the two last lines has been used. Summing up the same continuity relations for
the flow $\phi^{*}$ leads to the following inequality
\[
\sum_{\left(j,i\right)\in\overrightarrow{\partial}A_{k}}\phi_{ji}=-\sum_{i\in A_{k}}q_{i}<-\sum_{i\in A_{k}}q_{i}^{*}=\sum_{\left(j,i\right)\in\overrightarrow{\partial}A_{k}}\phi_{ji}^{*}.
\]
The above inequality implies that there exists $\left(v,w\right)\in\overrightarrow{\partial}A_{k}$
such that $\phi_{vw}^{*}>\phi_{vw}$. It remains to be checked that the
node $w\in A_{k}$ is an element of $B_{k}=A_{k}\setminus A_{k-1}$.
By construction if $w\in A_{k-1}$ then $v$ should be contained in
$A_{k}$ which is a contradiction. Thus $B_{k+1}$ is non-empty. We have just proved
by induction that $B_{k}$ for $k\leq n$ are non-empty and
moreover that $A_{k}$ is an increasing sequence of sets.
The last observation guarantees that $n$ is a finite number.

We now focus on proving that there exists a path $\left\{ i_{1},\ldots,i_{n}\right\} $
from some $t\in T$ to $u$ such that $\phi_{i_{l}i_{l+1}}^{*}>\phi_{i_{l}i_{l+1}}$.
By induction it is now easy to prove that for every node $i\in B_{k}$
there exists a path of length $k$ from $i$ to $u$ such that $\phi_{i_{l}i_{l+1}}^{*}>\phi_{i_{l}i_{l+1}}$.
Since it was already shown that there exists $n\in\mathbb{N}$ such that $B_{k}\cap T\neq\emptyset$
the proof is concluded.
\end{proof}
Theorem \ref{thm:aquarius} has an intuitive explanation,  clarifying the name chosen for the theorem. Picture
a flow of water through a network of aqueducts. If water leaks
appear at some points within the network the system becomes unbalanced.
Therefore an aqueduct inspector (``Aquarius'' in Latin) needs to compensate for the leaks pouring more water at
other nodes. This extra water added at the injection node will
flow along some path that has to end up at a leak. Otherwise the added
water accumulates and the network remains unbalanced.

\subsection{Dissipative Flow Networks}

Theorem \ref{thm:aquarius} applied to the dissipative flow
networks results in the following statement. 
\begin{cor}[Monotonicity of Potentials with Consumptions]
\label{cor:monotonicity_pressure_consumption}Let $\left(\phi,\pi\right)$
and $\left(\phi^{*},\pi^{*}\right)$ be solutions of the continuity
Eq.~\eqref{eq:flow_conservation} and the potential loss Eq.~\eqref{eq:potential_drop}
for the productions $q$ and $q^{*}$ respectively
and the same dissipation function $f$. If $\pi_{t}\geq\pi_{t}^{*}$
for all $t\in T\subset V$ and $q_{i}\geq q_{i}^{*}$ for all $i\in V\setminus T$
then $\pi_{u}\geq\pi_{u}^{*}$ for every node $u\in V\setminus T$. Moreover if $q_{u}>q_{u}^{*}$
the inequality is strict i.e. $\pi_{u}>\pi_{u}^{*}$.\end{cor}
\begin{proof}
Consider $u\in V\setminus T$.
One evaluates the potential drop Eq.~\eqref{eq:potential_drop} as explained by Theorem \ref{thm:aquarius} to arrive at the inequality
\begin{eqnarray*}
\pi_{u} & = & \sum_{k=1}^{n}-f\left(\phi_{i_{k}i_{k+1}}\right)+\pi_{t}\\
 & > & \sum_{k=1}^{n}-f\left(\phi_{i_{k}i_{k+1}}^{*}\right)+\pi_{t}^{*}\\
 & = & \pi_{u}^{*},
\end{eqnarray*}
where we have just used (in between the two last lines) that $f_{ij}$ is an increasing
function.
\end{proof}
Corollary \ref{cor:monotonicity_production_consumption} allows us to make the following statement about relations between productions and consumptions.
\begin{cor}[Monotonicity of Productions with Consumptions]
\label{cor:monotonicity_production_consumption}Let $\left(\phi,\pi\right)$
and $\left(\phi^{*},\pi^{*}\right)$ be solutions of the continuity
Eq.~\eqref{eq:flow_conservation} and the potential loss Eq.~\eqref{eq:potential_drop} for the productions $q$ and $q^{*}$ respectively and the same dissipation function $f$. If $\pi_{t}=\pi_{t}^{*}$
for all $t\in T\subset V$ and $q_{i}\geq q_{i}^{*}$ for all $i\in V\setminus T$
then $q_{t}\leq q_{t}^{*}$ for all $t\in T$.\end{cor}
\begin{proof}
Choose $u\in T$ and assume that $q_{u}>q_{u}^{*}$. Corollary \ref{cor:monotonicity_pressure_consumption}
applied to $u$ implies that $\pi_{u}>\pi_{u}^{*}$ which is a contradiction.
\end{proof}
Corollary \ref{cor:monotonicity_pressure_consumption} directly implies
that the maximum (and respectively minimum) of the potential $\pi_{i}\left(q_{R},q_{S},\pi_{T},f\right)$
given by Eq.~\eqref{eq:pressure_from_energy} is achieved at
a maximum (respectively minimum) of the customers production $q_{R}$
\begin{eqnarray}
\max_{q_{R}\in Q}\pi_{i}\left(q_{R},q_{S},\pi_{T},f\right) & = & \pi_{i}\left(\overline{q}_{R},q_{S},\pi_{T},f\right)\nonumber \\
\min_{q_{R}\in Q}\pi_{i}\left(q_{R},q_{S},\pi_{T},f\right) & = & \pi_{i}\left(\underline{q}_{R},q_{S},\pi_{T},f\right).\label{eq:max_pressure}
\end{eqnarray}
With regards to the Corollary \ref{cor:monotonicity_production_consumption}, it
tells us that the production $q_{i}\left(q_{R},q_{S},\pi_{T},f\right)$
given by Eq.~\eqref{eq:production_from_energy} achieves its
maximum (respectively minimum) at $\underline{q}_{R}$ (respectively at $\overline{q}_{R}$).
This implies that the maximum of our objective function is achieved
when all the regular customers are consuming at their maximum value
\begin{equation}
\min_{q_{R}\in Q}\sum_{i\in T}h_{i}\left(q_{i}\left(q_{R},q_{S},\pi_{T},f\right)\right)=\sum_{i\in T}h_{i}\left(q_{i}\left(\overline{q}_{R},q_{S},\pi_{T},f\right)\right).\label{eq:max_objective}
\end{equation}
Here in Eq.~(\ref{eq:max_objective}) we use the property that $h$ is a non-decreasing function of
$q_{i}$.

Eqs.~\eqref{eq:max_pressure} and \eqref{eq:max_objective}
show that only the two extreme scenarios of the regular customers
consumption has to be considered to guarantee feasibility of all other consumption configurations from the uncertain intervals.

\section{Illustration with Natural Gas Networks\label{sec:Illustration-with-Natural-gas-network}}

A natural gas network is a system of interconnected pipelines delivering natural gas from producers to consumers.  In the normal operational regime gas flows in the transmission (high pressure level) pipes are turbulent. Consider a pipeline of length $L$. As a pipeline length is much longer that its cross-section, the system can be modeled as one-dimensional, parameterized by position along the flow, $x\in[0,L]$, with cross-section effects averaged out. Then, the state
of the gas flow at a time $t$ is characterized by its pressure $p\left(x,t\right)$ and mass flow $\phi\left(x,t\right)$ along the pipe (both averaged over the cross-section). The two characteristics are related to each other via the following set of partial differential equations \cite{Osiadacz_1987,Misra2015,Thorley19873,sardanashvili2005computational}
\begin{eqnarray}
\partial_{t}\phi\left(x,t\right)+\partial_{x}p\left(x,t\right) & = & -\alpha\frac{\phi\left(x,t\right)\left|\phi\left(x,t\right)\right|}{2p\left(x,t\right)}+\kappa\left(x,t\right)\nonumber \\
\partial_{t}p\left(x,t\right)+c_s^2\partial_{x}\phi\left(x,t\right) & = & 0,\label{eq:gas_dynamic_equation}
\end{eqnarray}
where $c_s$ is the speed of sound in the gas and $\alpha$ is a constant that depends on the type of gas, size of the cross-section, roughness of the pipe surface and also on the Re-number, characterizing the level of turbulence.
The term $\kappa\left(x,t\right)$ in Eqs.~\eqref{eq:gas_dynamic_equation} accounts for a compression added at a compression station to compensate for pressure drop.
The first equation in Eqs.~\eqref{eq:gas_dynamic_equation} is
a phenomenological equation that quantify the loss of momentum due
to turbulent friction. The second equation in Eqs. \eqref{eq:gas_dynamic_equation} enforces conservation of the fluid mass along the pipe. Compression can be modeled as
\begin{equation}
\kappa\left(x,t\right)=b\delta\left(x-x_{c}\right),\label{eq:compression_station}
\end{equation}
where $\delta\left(\cdot\right)$ is the Dirac's delta function; $x_c$ is the compressor station position along the pipeline and $b$ is an additive compression factor of the station.

The equations are dynamic,  however for the purpose of planning the gas flow budget on the scale of a day, the dynamics in Eqs.~(\ref{eq:gas_dynamic_equation}) can be ignored, thus setting $\partial_{t}\phi=\partial_{t}p\equiv0$. In this case, after straightforward spatial integration along the pipe, Eq.~\eqref{eq:gas_dynamic_equation} transforms into
\begin{eqnarray}
p\left(L,t\right)^{2}-p\left(0,t\right)^{2} & = & -\frac{L\alpha}{2}\phi\left|\phi\right|+b\nonumber \\
\phi\left(x,t\right) & \equiv & \phi,\label{eq:steady_state_equation}
\end{eqnarray}
relating pressures at the ends of the pipe to the amount of flow and the value of compression acquired along the pipe.
Eq.~\eqref{eq:steady_state_equation} tells us that the flow in one pipe is constant and that it is driven by a difference of pressure squared at the endpoints of the pipeline. This static representation enables us to model the gas network in the steady-state regime as a dissipative network with a potential
equal to the pressure squared $\pi=p^{2}$. The dissipation function over a pipe $(i,j)$ is nonlinear:
\begin{equation}
f_{ij}\left(\phi_{ij}\right)=-\frac{L_{ij}\alpha_{ij}}{2}\phi_{ij}\left|\phi_{ij}\right|+b_{ij}.
\end{equation}

The three types of nodes that we have introduced above in the maximum profit problem
(sources, internal customers and terminal) map into nodes of the gas network as follows. The sources that inject gas into the network are gas producers, e.g. gas processing plants, Liquid Natural Gas (LNG) terminals and storage injecting gas into the system. The internal customers with uncertain demand are consumers with existing contracts for gas delivery, such as Local Distribution Companies (LDC) and electric gas-fired plants. The terminals are opportunistic customers ready to buy whatever amount of gas which can be made available (e.g. LNG terminals and storage reservoirs working in the regime of gas accumulation). The uncertainty on the side of the internal customers accounts for exogenous changes such as those related to LDC consumers' heating requirements, and fluctuations of gas consumption at the gas-fired plans due to uncertainty on the electric grid side (e.g. the renewable generation).

The choice of operation variables is not unique and depends on the regime of operation, type of the system, country, etc. For example in US,  typical large-scale producers would maintain constant injection/flow, thus allowing changes in the pressure, while LDC (and related city-gates) on the contrary would withdraw constant flow allowing the pressure to meander.

\section{Path Forward}
In this manuscript we proved that the robust maximum profit problem over nonlinear dissipative network flow problem is tractable. The strategy that we have employed in the proof is based on the search for an explicit formulation of the
special limiting scenarios such that feasibility of solution for the special scenarios guarantees feasibility for all other scenarios from the uncertainty range.  We proved that in the general case of the static dissipative network flow it is sufficient to maintain feasibility only for two scenarios.  The essence of this major step in our proof strategy is related to the very strong monotonicity property in the space of solutions.

We envision extending this work in the future along the following three directions:
\begin{itemize}
\item We plan on moving from static formulation to dynamic and thus to analyze dynamic versions of respective robust optimizations from the perspective of scenario reduction discussed in this manuscript. There are more than one possible generalization strategies (for transition from static to dynamic). In particular,  one may hope to find a dynamic (Lagrangian) generalization of the Aquarius principle, i.e. generalization of the Theorem \ref{thm:aquarius}. In the context of natural gas application, we plan to introduce and analyze consumption robust version of the dynamic optimization (off-line control) problem discussed in \cite{15SCB} based on the dynamic Eqs.~(\ref{eq:steady_state_equation}). We will also attempt to develop dynamic version of our scenario-reduction technique suitable for control and dynamic optimization traffic network problems of the type discussed in \cite{10CSADF,13Var_a,13Var_b}.
\item We would like to design an efficient numerical scheme to solve the tractable version of the robust maximum profit/minimum cost problem. As formulated in Theorem \ref{thm:ARC_max_throughput} the problem is a bi-level optimization task. To advance this task we plan to utilize the energy function representation in order to formulate the entire problem as one minimization procedure in the spirit of \cite{Misra2015a}.
\item Finally, we plan to extend the results reported in the manuscript to more general types of uncertainty sets such as ellipsoids. The ellipsoid type of uncertainty set is more challenging than the one of the box kind considered in the manuscript. Even if the monotonicity properties guarantee existence of the extremal scenarios, deriving explicit form of the extremal scenarios in setting other than of the box type remains a challenge. 
\end{itemize}

\section*{Acknowledgment}

The authors thank S. Backhaus for multiple discussions and advice.
The work at LANL was carried out under the auspices of the National Nuclear Security Administration of the U.S. Department of Energy at Los Alamos National Laboratory under Contract No. DE-AC52-06NA25396 and it was partially supported by DTRA Basic Research Project $\#10027-13399$. The authors also acknowledge partial support of the Advanced Grid Modeling Program in the US Department of Energy Office of Electricity.

\bibliographystyle{IEEEtran}
\bibliography{cdc_gas_network,cdc_robust_sm}

\end{document}